\documentclass[final]{siamltex}

\usepackage{amsfonts,amssymb,amsmath,cmmib57,color,graphics,mathdots}
\usepackage[latin1]{inputenc}

\newcommand{\FF}{{\mathbb F}}
\newcommand{\CC}{{\mathbb C}}

\newsavebox{\savepar}

\begin{document}

\title{A note on the consistency of a system of $\star$-Sylvester equations}

\author{Fernando De Terán\thanks{Department of Mathematics, Universidad Carlos III de Madrid, Avenida de la Universidad 30, 28911, Legan\'es, Spain ({\tt fteran@math.uc3m.es})}}

\maketitle

\begin{abstract}
Let $\FF$ be a field with characteristic not $2$, and $A_i\in\FF^{m\times n},B_i\in\FF^{n\times m},C_i\in\FF^{m\times m}$, for $i=1,\hdots,\ell$. In this short note, we obtain necessary and sufficient conditions for the consistency of the system of $\star$-Sylvester equations $A_iX-X^\star B_i=C_i$, for $i=1,\hdots,\ell$, where $\star$ denotes either the transpose or, when $\FF=\CC$, the conjugate transpose.
\end{abstract}

\begin{keywords}
Sylvester equation, Roth's theorem, congruence, system of matrix equations.
\end{keywords}

\begin{AMS}
15A24, 15A30.
\end{AMS}


The {\em Sylvester equation} $AX-XB=C$, where $A,B,C$ are matrices of the appropriate size, is one of the matrix equations that most frequently appear in the literature. This equation has been widely studied, from both the theoretical and the numerical point of view. In particular, necessary and sufficient conditions for consistency are know for a long time \cite{roth}. These conditions are usually known as `Roth's Theorem', which states that $AX-XB=C$ is consistent if and only if there is a nonsingular matrix $S$ such that
\begin{equation}\label{consistency-sylv}
S\left[\begin{array}{cc}A&C\\0&B\end{array}\right]S^{-1}=\left[\begin{array}{cc}A&0\\0&B\end{array}\right]
\end{equation}

More recently, the $\star$-Sylvester equation $AX-X^\star B=C$ has become of interest, because of its connection with some applied problems, from areas like the theory of congruence orbits of matrices (when $B=-A$ and $C=0$), or the eigenvalue problem of palindromic matrix pencils (see, for instance, \cite{dd-ELA} and the references therein). In this equation, the notation $(\cdot)^\star$ stands for both the transpose of a matrix or, in the case of matrices over the complex field, the conjugate transpose. Necessary and sufficient conditions for consistency of the $\star$-Sylvester equation are also known \cite{dd-ELA,wimmer94}. More precisely, $AX-X^\star B=C$ is consistent if and only if there is a nonsingular matrix $S$ such that
\begin{equation}\label{consistency-star}
S\left[\begin{array}{cc}C&-A\\B&0\end{array}\right]S^\star=\left[\begin{array}{cc}0&-A\\B&0\end{array}\right].
\end{equation}

It is worth to emphasize that the necessary and sufficient conditions \eqref{consistency-sylv} and \eqref{consistency-star} look like very much to each other. The main difference between them is that the role played by the similarity transformation in the case of the Sylvester equation \eqref{consistency-sylv} is played by the congruence transformation in the $\star$-Sylvester equation \eqref{consistency-star}. For this reason, \eqref{consistency-star} can be also termed as `Roth's theorem' for the $\star$-Sylvester equation \cite{wimmer94}.

It is natural to consider, as an immediate generalization, instead of a single Sylvester (respectively, $\star$-Sylvester) equation, a system of equations $A_iX-XB_i=C_i$ (resp. $A_iX-X^\star B_i=C_i$), for $i=1,\hdots,\ell$. For these systems, it is immediate to get a sufficient condition for consistency, namely, that equation \eqref{consistency-sylv} (resp. \eqref{consistency-star}) holds after replacing $A,B,C$ by $A_i,B_i,C_i$, respectively, for all $i=1,\hdots,\ell$, and with the same matrix $S$ (namely, that all block-partitioned matrices are simultaneously similar or simultaneously congruent). This is because the solution of the system $X$ gives a particular $S$ for which \eqref{consistency-sylv} (resp. \eqref{consistency-star}) holds for all $A_i,B_i,C_i$ (see the implication ``(a) $\Rightarrow$ (b)" in the proof of Theorem \ref{main_th} for \eqref{consistency-star}). However, to show that this condition is necessary is not so immediate, and requires some additional effort. In particular, that this condition is necessary and sufficient for a system of Sylvester equations has been proved in \cite{guralnick} for matrices over commutative rings and, more recently, in \cite{lee-vu}  for matrices over fields, using a completely different approach, which is based on the proof of Theorem 2 in \cite{wimmer94}. In this short note we also follow the arguments in the proof of Theorem 2 in \cite{wimmer94} to extend the characterization of consistency \eqref{consistency-star} of a single $\star$-Sylvester equation to a system of $\star$-Sylvester equations. This characterization is stated in Theorem \ref{main_th}. In the following, $I_n$ denotes the $n\times n$ identity matrix. 

\begin{theorem}\label{main_th} Let $\FF$ be a field with characteristic not $2$, and $A_i\in\FF^{m\times n},B_i\in\FF^{n\times m},C_i\in\FF^{m\times m}$, for $i=1,\hdots,\ell$. Then, the following statements are equivalent:

\begin{itemize}

\item[{\rm(a)}] The system of $\star$-Sylvester equations
\begin{equation}\label{system}
\left\{\begin{array}{ccc}A_1X-X^\star B_1&=&C_1\\\vdots&&\vdots\\A_\ell X-X^\star B_\ell&=&C_\ell\end{array}\right.
\end{equation}
has a solution, $X\in\FF^{n\times m}$.

\item[{\rm(b)}] There is an invertible matrix $S\in\FF^{(m+n)\times(m+n)}$ such that
\begin{equation}\label{congruence}
S\left[\begin{array}{cc}C_i&-A_i\\B_i&0\end{array}\right]S^\star=\left[\begin{array}{cc}0&-A_i\\B_i&0\end{array}\right],\qquad\mbox{for all $i=1,\hdots,\ell$.}
\end{equation}
\end{itemize}

\end{theorem}

\begin{proof} (a) $\Rightarrow$ (b): If there is an $X\in\FF^{n\times m}$ such that \eqref{system} holds, then set
$$
S=\left[\begin{array}{cc}I_n&X^\star\\0&I_m\end{array}\right].
$$
Note that $S$ is invertible. It is straightforward to check that \eqref{congruence} holds for this $S$.

(b) $\Rightarrow$ (a): The proof follows the one of ``(c) $\Rightarrow$ (a)" in Theorem 2 in \cite{wimmer94}, with the appropriate changes. 

Let us define the following vector subspaces of $\FF^{(m+n)\times(m+n)}\times \FF^{(m+n)\times(m+n)}$:
$$
\begin{array}{c}
\Gamma(A_i,B_i,C_i):=\left\{(U,W)\ :\ \left[\begin{array}{cc}0&-A_i\\B_i&0\end{array}\right]U+W\left[\begin{array}{cc}C_i&-A_i\\B_i&0\end{array}\right]=0\right\},
\\
\Delta(A_i,B_i,C_i):=\left\{(U,W)\ :\ U^\star\left[\begin{array}{cc}0&-A_i\\B_i&0\end{array}\right]+\left[\begin{array}{cc}C_i&-A_i\\B_i&0\end{array}\right]W^\star=0 \right\},
\end{array}
$$
and $D(A_i,B_i,C_i):=\Gamma(A_i,B_i,C_i)\cap\Delta(A_i,B_i,C_i).$ Let us partition
\begin{equation}\label{partition}
U=\left[\begin{array}{cc}U_{11}&U_{12}\\U_{21}&U_{22}\end{array}\right],\qquad\mbox{and}\qquad
W=\left[\begin{array}{cc}W_{11}&W_{12}\\W_{21}&W_{22}\end{array}\right],
\end{equation}
where $U_{11},W_{11}\in\FF^{m\times m}$, and $U_{22},W_{22}\in\FF^{n\times n}$. Then 
\begin{equation}\label{set1}
(U,W)\in\Gamma(A_i,B_i,C_i)\Leftrightarrow
\left\{\begin{array}{c}A_iU_{21}-W_{12}B_i-W_{11}C_i=0\\B_iU_{11}+W_{22}B_i+W_{21}C_i=0\\A_iU_{22}+W_{11}A_i=0\\B_iU_{12}-W_{21}A_i=0\end{array}\right.,
\end{equation}
and
\begin{equation}\label{set2}
(U,W)\in\Delta(A_i,B_i,C_i)\Leftrightarrow
\left\{\begin{array}{c}A_iW_{12}^\star-U_{21}^\star B_i-C_iW_{11}^\star=0\\U_{11}^\star A_i+A_iW_{22}^\star-C_iW_{21}^\star=0\\U_{22}^\star B_i+B_iW_{11}^\star=0\\U_{12}^\star A_i-B_iW_{21}^\star=0\end{array}\right..
\end{equation}

Now, the problem reduces to prove that there is a pair $(U,W)\in D:=\bigcap_{i=1}^\ell D(A_i,B_i,C_i) $ with $W_{11}=I_m$. Note that, for such a pair, the first identities in \eqref{set1} and \eqref{set2} give
$$
\begin{array}{c}
A_iU_{21}-W_{12}B_i=C_i\\
A_iW_{12}^\star-U_{21}^\star B_i=C_i,
\end{array}\qquad\mbox{for $i=1,\hdots,\ell$},
$$ 
so, adding up, we get $A_i(U_{21}+W_{12}^\star)-(U_{21}+W_{12}^\star)^\star B_i=2C_i$, hence $X=\frac{1}{2}(U_{21}+W_{12}^\star)$ is a solution of \eqref{system}.

Using the partition \eqref{partition}, let us define the following linear map:
$$
\begin{array}{cccc}
\varphi:&\FF^{(m+n)\times(m+n)}\times\FF^{(m+n)\times(m+n)}&\longrightarrow &\FF^{(m+n)\times(n)}\\
&(U,W)&\mapsto&\left[\begin{array}{c}W_{11}\\W_{21}\end{array}\right].
\end{array}
$$
Then, it suffices to prove that 
\begin{equation}\label{image}
\left[\begin{array}{c}I_m\\0\end{array}\right]\in\varphi(D).
\end{equation}

Now, set $\widehat\varphi:=\varphi|_D$ and $\varphi_0:=\varphi|_{D_0}$, where $D_0:=\bigcap_{i=1}^\ell D(A_i,B_i,0)$. Let us assume that the following four claims hold:
\begin{itemize}

\item[(i)] $\dim D=\dim D_0$.

\item[(ii)] Ker$\,\widehat \varphi=$ Ker$\,\varphi_0$.

\item[(iii)] Im$\,\widehat \varphi\subseteq$  Im$\,\varphi_0$.

\item[(iv)] $\left[\begin{array}{c}I_m\\0\end{array}\right]\in\varphi(D_0).$
\end{itemize}

Form (i)--(ii) and the identities:
$$
\begin{array}{c}
\dim \mbox{\rm Ker}\, \widehat \varphi+\dim\mbox{\rm Ker}\,\widehat\varphi=\dim D,\\
\dim \mbox{\rm Ker}\, \varphi_0+\dim\mbox{\rm Ker}\,\varphi_0=\dim D_0,
\end{array}
$$
it follows that $\dim\mbox{\rm Im}\,\widehat \varphi=\dim\mbox{\rm Im}\,\varphi_0$. Then (iii) implies $\mbox{\rm Im}\,\widehat\varphi=\mbox{\rm Im}\,\varphi_0$, and (iv) implies \eqref{image}. Hence, it remains to prove (i)--(iv).

\begin{itemize}

\item Proof of (iv): For this, just notice that $(-I_{m+n},I_{m+n})\in D_0$, and that $\varphi(-I_{m+n},I_{m+n})=\left[\begin{array}{c}I_m\\0\end{array}\right]$. 

\item Proof of (ii): Just notice that the coefficients of $C_i$ in both \eqref{set1} and \eqref{set2} are $W_{11}$ or $W_{21}$, which are precisely the ones appearing in $\varphi(U,W)$.

\item Proof of (iii): If we set 
$$
\widetilde U=\left[\begin{array}{cc}0&U_{12}\\0&U_{22}\end{array}\right]\qquad\mbox{and}\qquad
\widetilde W=\left[\begin{array}{cc}W_{11}&0\\W_{21}&0\end{array}\right],
$$
then, as can be seen by looking at \eqref{set1} and \eqref{set2}, $(U,W)\in D$ implies $(\widetilde U,\widetilde W)\in D_0$.

\item Proof of (i): Let $S\in\FF^{(m+n)\times(m+n)}$ satisfying \eqref{congruence}. Then $(U,W)\in D$ if and only if $(US^\star,WS^{-1})\in D_0$. To prove this, notice that:

$$
\begin{array}{cl}(US^\star,WS^{-1})\in\Gamma(A_i,B_i,0)&\Leftrightarrow\left[\begin{array}{cc}0&-A_i\\B_i&0\end{array}\right]US^\star+WS^{-1}\left[\begin{array}{cc}0&-A_i\\B_i&0\end{array}\right]=0
\\ &\Leftrightarrow\left[\begin{array}{cc}0&-A_i\\B_i&0\end{array}\right]US^\star+W\left[\begin{array}{cc}C_i&-A_i\\B_i&0\end{array}\right]S^\star=0\\
&\Leftrightarrow \left[\begin{array}{cc}0&-A_i\\B_i&0\end{array}\right]U+W\left[\begin{array}{cc}C_i&-A_i\\B_i&0\end{array}\right]=0\\&\Leftrightarrow (U,W)\in\Gamma(A_i,B_i,C_i).\end{array}
$$ 

$$
\begin{array}{cl}(US^\star,WS^{-1})\in\Delta(A_i,B_i,0)&\Leftrightarrow SU^\star \left[\begin{array}{cc}0&-A_i\\B_i&0\end{array}\right]+\left[\begin{array}{cc}0&-A_i\\B_i&0\end{array}\right](S^{-1})^\star W^\star=0
\\ &\Leftrightarrow SU^\star\left[\begin{array}{cc}0&-A_i\\B_i&0\end{array}\right]+S\left[\begin{array}{cc}C_i&-A_i\\B_i&0\end{array}\right]W^\star=0\\
&\Leftrightarrow U^\star\left[\begin{array}{cc}0&-A_i\\B_i&0\end{array}\right]+\left[\begin{array}{cc}C_i&-A_i\\B_i&0\end{array}\right]W^\star=0\\&\Leftrightarrow (U,W)\in\Delta(A_i,B_i,C_i).\end{array}
$$ 

\end{itemize}
\end{proof}


\begin{thebibliography}{99}

\bibitem{dd-ELA} F.\ De Terán, F.\ M.\ Dopico.
\newblock Consistency and efficient solution of the Sylvester equation for $\star$-congruence.
\newblock {\em Electron. J. Linear Algebra}, 22:849--863 (2011). 

\bibitem{guralnick} R.\ M.\ Guralnick.
\newblock Roth's theorem for sets of matrices.
\newblock {\em Linear Algebra Appl.}, 71:113--117 (1985).

\bibitem{lee-vu} S.-G.\ Lee, Q.-P.\ Vu.
\newblock Simultaneous solution of matrix equations and simultaneous equivalence of matrices.
\newblock {\em Linear Algebra Appl.}, 437:2325--2339 (2012). 

\bibitem{roth} W.\ F.\ Roth.
\newblock The equations $AX-YB=C$ and $AX-XB=C$ in matrices.
\newblock {\em Proc. Amer. Math. Soc.}, 3:392--396 (1952). 

\bibitem{wimmer94} H.\ K.\ Wimmer.
\newblock Roth's theorem for matrix equations with symmetry constraints.
\newblock {\em Linear Algebra Appl.}, 199:357--362 (1994)

\end{thebibliography}
\end{document}